\documentclass[12pt,reqno]{amsart}
\usepackage{amsthm}
\usepackage{amssymb}
\usepackage{graphics}
\usepackage{tikz}
\usetikzlibrary{shapes,backgrounds,calc}
\usepackage{latexsym}
\usepackage{multicol}
\usepackage{verbatim,enumerate}
\usepackage{accents}
\usepackage{cite}
\usepackage{multirow}
\usepackage{bigstrut}
\usepackage{amsthm}
\usepackage{amssymb}
\usepackage{graphics}
\usepackage{tikz}
\usetikzlibrary{shapes,backgrounds,calc}
\usepackage{latexsym}
\usepackage{multicol}
\usepackage{verbatim,enumerate}
\usepackage{accents}
\usepackage{cite}
\usepackage{array}
\usepackage[colorlinks=true, linkcolor=blue, citecolor=blue, urlcolor=black]{hyperref}
\usepackage{hyperref}
\usepackage{amsmath, amscd,url}
\usepackage{multirow}
\usepackage{bigstrut}
\usepackage{longtable}

\usepackage{stackengine}

\newcommand{\ind}{\mathop{\mathrm{ind}}}

\newcommand{\F}{\mbox{$\mathbb F$}}	

\usepackage{hyperref}
\usepackage{amsmath, amscd,url}
\usepackage{longtable}

\advance\textwidth by 1.3in \advance\oddsidemargin by -.6in \advance\evensidemargin by -.6in
\theoremstyle{definition}

\newtheorem{theorem}{Theorem}[section]
\newtheorem{lemma}[theorem]{Lemma}
\newtheorem{corollary}[theorem]{Corollary}
\theoremstyle{definition}
\newtheorem{definition}[theorem]{Definition}
\newtheorem{example}[theorem]{Example}

\theoremstyle{remark}
\newtheorem{remark}[theorem]{Remark}

\theoremstyle{definition}

\newcounter{cnt}
 \makeatletter
\def\mydggeometry{\makeatletter\dg@YGRID=1\dg@XGRID=20\unitlength=0.003pt\makeatother}
\makeatother \theoremstyle{remark}

\numberwithin{equation}{section}
\let\bwdg\bigwedge
\def\bigwedge{{\textstyle\bwdg}}

\newcommand{\Q}{\mathbb{Q}}
\newcommand{\Z}{\mathbb{Z}}

\newcommand{\nc}{\newcommand}
\newcommand{\rnc}{\renewcommand}

\nc{\cal}{\mathcal} \nc{\goth}{\mathfrak} \rnc{\bold}{\mathbf}

\nc\bomega{{\mbox{\boldmath $\omega$}}} \nc\bpsi{{\mbox{\boldmath $\Psi$}}}
 \nc\balpha{{\mbox{\boldmath $\alpha$}}}
 \nc\bpi{{\mbox{\boldmath $\pi$}}}
 \nc\bvpi{{\mbox{\boldmath $\varpi$}}}
\nc\chara{\operatorname{ch}}

  \nc\bxi{{\mbox{\boldmath $\xi$}}}
\nc\bmu{{\mbox{\boldmath $\mu$}}} \nc\bcN{{\mbox{\boldmath $\cal{N}$}}} \nc\bcm{{\mbox{\boldmath $\cal{M}$}}} \nc\blambda{{\mbox{\boldmath
$\lambda$}}}\nc\bnu{{\mbox{\boldmath $\nu$}}}

\makeatletter
\def\section{\def\@secnumfont{\mdseries}\@startsection{section}{1}%
  \z@{.7\linespacing\@plus\linespacing}{.5\linespacing}%
  {\normalfont\scshape\centering}}
\def\subsection{\def\@secnumfont{\bfseries}\@startsection{subsection}{2}%
  {\parindent}{.5\linespacing\@plus.7\linespacing}{-.5em}%
  {\normalfont\bfseries}}
\makeatother

 \nc{\Hom}{\operatorname{Hom}}
  \nc{\mode}{\operatorname{mod}}
\nc{\End}{\operatorname{End}} \nc{\wh}[1]{\widehat{#1}} \nc{\Ext}{\operatorname{Ext}} \nc{\ch}{\text{ch}} \nc{\ev}{\operatorname{ev}}
\nc{\Ob}{\operatorname{Ob}} \nc{\soc}{\operatorname{soc}} \nc{\rad}{\operatorname{rad}} \nc{\head}{\operatorname{head}}

 \nc{\Cal}{\cal} \nc{\Xp}[1]{X^+(#1)} \nc{\Xm}[1]{X^-(#1)}

\nc{\N}{{\bold N}}  \nc\boa{\bold a} \nc\bob{\bold b} \nc\boc{\bold c} \nc\bod{\bold d} \nc\boe{\bold e} \nc\bof{\bold f} \nc\bog{\bold g}
\nc\boh{\bold h} \nc\boi{\bold i} \nc\boj{\bold j} \nc\bok{\bold k} \nc\bol{\bold l} \nc\bom{\bold m} \nc\bon{\mathbb n} \nc\boo{\bold o}
\nc\bop{\bold p} \nc\boq{\bold q} \nc\bor{\bold r} \nc\bos{\bold s} \nc\boT{\bold t} \nc\boF{\bold F} \nc\bou{\bold u} \nc\bov{\bold v}
\nc\bow{\bold w} \nc\boz{\bold z}\nc\ba{\bold A} \nc\bb{\bold B} \nc\bc{\mathbb C} \nc\bd{\bold D} \nc\be{\bold E} \nc\bg{\bold
G} \nc\bh{\bold H} \nc\bi{\bold I} \nc\bj{\bold J} \nc\bk{\bold K} \nc\bl{\bold L} \nc\bm{\bold M} \nc\bn{\mathbb N} \nc\bo{\bold O} \nc\bp{\bold
P} \nc\bq{\bold Q} \nc\br{\bold R} \nc\bs{\bold S} \nc\bt{\bold T} \nc\bu{\bold U} \nc\bv{\bold V} \nc\bw{\bold W} \nc\bz{\mathbb Z} \nc\bx{\bold
x} \nc\KR{\bold{KR}} \nc\rk{\bold{rk}} \nc\het{\text{ht }}

\nc\toa{\tilde a} \nc\tob{\tilde b} \nc\toc{\tilde c} \nc\tod{\tilde d} \nc\toe{\tilde e} \nc\tof{\tilde f} \nc\tog{\tilde g} \nc\toh{\tilde h}
\nc\toi{\tilde i} \nc\toj{\tilde j} \nc\tok{\tilde k} \nc\tol{\tilde l} \nc\tom{\tilde m} \nc\ton{\tilde n} \nc\too{\tilde o} \nc\toq{\tilde q}
\nc\tor{\tilde r} \nc\tos{\tilde s} \nc\toT{\tilde t} \nc\tou{\tilde u} \nc\tov{\tilde v} \nc\tow{\tilde w} \nc\toz{\tilde z} \nc\woi{w_{\omega_i}}

\begin{document}
\setcounter{section}{0}
\setcounter{tocdepth}{1}


\title{On common index divisor of the number fields defined by   $x^7+ax+b$}

\author[Anuj Jakhar]{Anuj Jakhar}
\author[Sumandeep Kaur]{Sumandeep Kaur}
\author[Surender Kumar]{Surender Kumar}
\address[Anuj Jakhar]{Department of Mathematics, Indian Institute of Technology (IIT) Madras}
\address[Sumandeep Kaur]{Department of Mathematics, Panjab University Chandigarh}
\address[Surender Kumar]{Department of Mathematics, Indian Institute of Technology (IIT) Bhilai}

\thanks{The first author is thankful to SERB grant SRG/2021/000393 and IIT Madras for NFIG  RF/22-23/1035/MA/NFIG/009034. The second author is grateful to the Council of Scientific and Industrial Research, New Delhi for providing financial support in the form of Senior Research Fellowship through Grant No. $09/135(0878)/2019$-EMR-$1$. The third author is grateful to the University Grants Commision, New Delhi for providing financial support in the form of Junior Research Fellowship through Ref No.1129/(CSIR-NET JUNE 2019).}



\begin{center}

\end{center}

\begin{abstract}
\noindent  Let $f(x)=x^7+ax+b$ be an irreducible polynomial having integer coefficients and   $K=\mathbb{Q}(\theta)$ be an algebraic number field generated by a  root $\theta$ of $f(x)$. In the present paper, for every rational prime $p$, our objective is to  determine the necessary and sufficient conditions involving only $a,~b$ so that $p$ is a divisor of  the index of the field $K$. In particular, we provide sufficient conditions on $a$ and $b$, for which $K$ is non-monogenic. In a special case, we show that if either $8$ divides both $a\pm1$, $b$ or $32$ divides both $a+4$, $b$, then  $K$ is non-monogenic. We illustrate our results through examples.
\end{abstract}
\maketitle

\section{Introduction and statements of results}\label{intro}

Let $K=\Q(\theta)$ be an algebraic number field with $\theta$ in the ring $\mathbb{O}_K$ of algebraic integers of $K$. Let $f(x)\in\Z[x]$  be the minimal polynomial of $\theta$ having degree $n$ over the field $\Q$ of rational numbers.   It is a basic result in algebraic number theory  that $\mathbb{O}_K$ is a free abelian group  of rank $n$.   An algebraic number field $K$ is said to be  monogenic if there exists some  $\alpha \in \mathbb{O}_K$ such that $\{1, \alpha,\cdots,\alpha^{n-1}\}$ is an integral basis of $K$.  In this case $\mathbb{O}_K=\Z[\alpha]$, i.e., $[\mathbb{O}_K:\Z[\alpha] ]=1$. If $K$ does not have any such $\alpha$, then the field $K$ is  said to be non-monogenic. It is well-known that every quadratic and cyclotomic field is monogenic.  It is important to know that whether a number field is monogenic or not.  It was Dedekind, who gave the first non-monogenic number field $K=\mathbb{Q}(\xi)$, where $\xi$ is a root of the polynomial $x^3-x^2-2x-8$ (cf. \cite[page 64]{Nar}).  The problem of testing the monogenity of number
fields and constructing power integral bases have been intensively studied (cf. \cite{ANH1}, \cite{ANH}, \cite{GR},  \cite{JK}, \cite{TM}, \cite{SJ}, \cite{MS}, \cite{SS}).  In $1984$, Funakura \cite{Fun} gave necessary and sufficient conditions on those integers $m$ for which  the quartic field $\Q(m^{1/4})$ is monogenic. Ahmad, Nakahara and Husnine in  
 \cite{ANH1}, \cite{ANH} proved that for a square free integer $m$ not congruent to $\pm 1\mod 9$, a pure field $\Q(m^{1/6})$ having degree six over $\Q$ is monogenic when $m\equiv 2$ or $3\mod 4$ and it is non-monogenic when $m\equiv 1\mod 4$.
 In $2017$,  Ga\'{a}l and Remete \cite{GR} studied monogenity of  algebraic number fields of the type $\Q(m^{1/n})$ where  $3\leq n\leq 9$ and $m$ is square free.  In \cite{JK}, A. Jakhar and S. Kumar provides some sufficient conditions for which the  number field $K$ defiend by the sextic trinomial $x^6+ax+b\in \Z[x],$ is non-monogenic.
  In \cite{GA}, Ga\'al studied monogenity of number fields defined by some sextic irreducible trinomials. 

  Recall that for an algebraic number field $L=\Q(\xi)$ with $\xi$ an algebraic integer satisfying a monic irreducible  polynomial $g(x)$ over  $\Q$,  the discriminant $D$ of $g(x)$  and the discriminant $d_L$ of $L$ are related by the formula \begin{equation}\label{eq1}
 D=(\ind \xi)^2\cdot d_L. 
 \end{equation}
\indent Throughout this paper, $\ind \theta$ will  denote the index of the subgroup $\mathbb{Z}[\theta]$ in $\mathbb{O}_K $ and   $i(K)$ will stand for the  index of the field $K$ defined by $$ i(K) = \gcd\{\ind \alpha \mid  {\text  {} K=\mathbb{Q}(\alpha)  {\text { and }}      \alpha\in \mathbb{O}_K} \}.$$ A prime divisor of  $i(K)$ is called a common index divisor of $K$. Note that  $i(K)=1$, for every monogenic number field $K$. But there exist  non-monogenic number fields having $i(K)=1$, e.g., $K=\Q(\sqrt[3]{175})$ is non-monogenic with $i(K)=1$.\\
\indent In what follows, let $K=\Q(\theta)$ be an algebraic number field  with $\theta$ a root of an irreducible trinomial $f(x)=x^7+ax+b\in\Z[x]$, then for every rational prime $p$, we provide necessary and sufficient conditions on $a,~b$, so that $p$ is a common index divisor of $K$. In particular, under these conditions $K$ is non-monogenic. 
  Our method  is  based on the  theory of Newton polygons and theorem of Ore. In $1878$, (see [\cite {ZA},
  Theorem 6.1.4]) Dedekind  gave  a criterion to determine whether $p$
  divides the index  $[\mathbb{O}_K:\Z[\theta]]$. When $p$ divides  the index $ [\mathbb{O}_K:\Z[\theta]]$, then a method of Ore 1928, can be used in order to evaluate the prime ideal factorization of $p\mathbb{O}_K$
 (see \cite{ZB}, \cite{ZC}). If Ore's method doesn't work, then an
  algorithm developed by Guardia, Montes, and Nart \cite{ZD}, based on higher order Newton
  polygons can be used to determine the prime ideal
  factorization of $p\mathbb{O}_K$.

  For a prime $p$ and a non-zero $m$ belonging  to the ring $\mathbb{Z}_p$  of $p$-adic integers,  $v_p(m)$ will denote the highest power of $p$ dividing $m$.  For a non-zero integer $l$, let $l_p$ denote $\frac{l}{p^{v_p(l)}}$.  If  a rational prime $p$ is such that $p^6$ divides $ a$ and  $p^7$ divides $ b$, then $\theta/p$ is a root of the polynomial $x^7+(a/p^6)x+(b/p^7)$ having integer coefficients. So we may assume that for each prime $p$
  \begin{equation}\label{1}
  \text{either} ~v_p(a)\leq 5~{ \text{or}}~ v_p(b)\leq 6.
  \end{equation}  Also, $D$ will stand for the discriminant of $f(x)$. One can check that 
  \begin{equation}\label{eq2}
  D=-7^7b^6-6^6a^7.
  \end{equation}

  With the above notations and assumption (\ref{1}), we prove
  \begin{theorem}\label{Th1.1}
  	Let $K=\Q(\theta)$ be an algebraic number field with $\theta$ a root of an irreducible polynomial $f(x)=x^7+ax+b $. If $u=\frac{v_2(D)-6}{2}$ and  $\rho=2^{u}+\frac{7b}{6a},$ then $2\mid i(K)$ if and only if one of the following hold:
  	\begin{enumerate}
  		
  		\item $a\equiv 3\mod4$ and $b\equiv 0\mod8$.
  		\item  $a\equiv 3\mod8$ and $b\equiv 4\mod8$.
  		\item $a\equiv 1\mod 4$ and $b\equiv 0\mod 4.$
  		\item $a\equiv 1\mod4$, $b\equiv 2\mod4$, $v_2(D)$ is odd.
  		\item $a\equiv 1\mod4$, $b\equiv 2\mod4$, $v_2(D)$ is even and $v_2(f(\rho))=2u+1$.
  		\item $a\equiv 1\mod4$, $b\equiv 2\mod4$, $v_2(D)$ is even and $v_2(f(\rho))\ge2u+3$.
  		\item $a\equiv 28\mod32$ and $b\equiv 0\mod32$.
  		\item $a\equiv 48\mod64$ and  $b\equiv 0\mod 128$.
  	\end{enumerate}
  
  The next three corollaries follow immediately from the above theorem.
  \begin{corollary}\label{cor1}
  	Let $K=\Q(\theta)$ be an algebraic number field, where $\theta$  satisfies the  irreducible polynomial $x^7+ax+b\in \Z[x]$.   If $8$ divides both $a\pm1$ and $b$, then   $K$ is non-monogenic.
  	
  \end{corollary}
\begin{corollary}\label{cor1}
	Let $K=\Q(\theta)$ be an algebraic number field generated by a root  $\theta$ of an irreducible polynomial $x^7+ax+b$ belonging to $\Z[x]$.   If  both $a+4$ and $b$  are divisible by $32$, then $K$ is non-monogenic.
	
\end{corollary}
\begin{corollary}\label{cor1}
	Let $K=\Q(\theta)$ be an algebraic number field generated by a root  $\theta$ of an irreducible polynomial $x^7+ax+b$ belonging to $\Z[x]$. Then $K$ is non-monogenic,  if  both $a+16$ and $b$ are divisible by $128$.
	
\end{corollary}

  \end{theorem}

 \begin{theorem}\label{Th1.2}
Let $K=\Q(\theta)$ be an algebraic number field with $\theta$ a root of an irreducible polynomial $f(x)=x^7+ax+b $. Then $3\mid i(K)$ if and only if one of the following hold:
\begin{enumerate}
	\item $a\equiv 2\mod 9$, $v_3(b)=1$, $b_3\equiv 1\mod 3$, $v_3(a+7)=2$ and $v_3(b-a-1)\ge 3$.
	\item  $a\equiv 2\mod 9$, $v_3(b)=1$, $b_3\equiv 1\mod 3$, $v_3(a+7)\ge3$, $v_3(b-a-1)\ge 3$, $2v_3(a+7)>v_3(b-a-1)+1$ and $(b-a-1)_3\equiv -1\mod 3.$
	\item $a\equiv 2\mod 9$, $v_3(b)=1$,  $b_3\equiv 1\mod 3$, $v_3(a+7)\ge3$, $v_3(b-a-1)\ge 3$ and $2v_3(a+7)<v_3(b-a-1)+1$.
	\item $a\equiv 2\mod 9$, $v_3(b)=1$,  $b_3\equiv -1\mod 3$, $v_3(a+7)=2$ and $v_3(b+a+1)\ge 3$.
	\item $a\equiv 2\mod 9$, $v_3(b)=1$,  $b_3\equiv -1\mod 3$, $v_3(a+7)\ge3$, $v_3(b+a+1)\ge 3$, $2v_3(a+7)>v_3(a+b+1)+1$ and $(b+a+1)_3\equiv -1\mod 3.$
	\item $a\equiv 2\mod 9$, $v_3(b)=1$,  $b_3\equiv -1\mod 3$, $v_3(a+7)\ge3$, $v_3(b+a+1)\ge 3$ and $2v_3(a+7)=v_3(a+b+1)+1$.
	\item  $a\equiv 5\mod 9$ and $v_3(b)=1$.
	\item $v_3(a+1)>1$, $v_3(b)>1$ and  $((a+1)_3,b_3)\in \{(-1,1), (1,-1)\}\mod3$.
	\item $v_3(a+1)>1$, $v_3(b)>1$  and $((a+1)_3,b_3)\in \{(-1,-1), (1,1)\}\mod3$.

\end{enumerate}

\end{theorem}

 The following  corollary follow immediately from the above theorem.
   \begin{corollary}\label{cor2}
         Let $K=\Q(\theta)$ where $\theta$  satisfies the  irreducible polynomial $x^7+ax+b$. Then  $K$ is non-monogenic, if one of the following hold:
         \begin{enumerate}
        \item $a\equiv 5\mod 9$ and $b\equiv 3 \mod 9$.
         \item $a\equiv 5\mod 9$ and $b\equiv 6 \mod 9$.
        \end{enumerate}
         \end{corollary}
     \begin{theorem}\label{Th1.3}
     	Let $K=\Q(\theta)$ be an algebraic number field with $\theta$ a root of an irreducible polynomial $f(x)=x^7+ax+b\in \Z[x]$. Let  $p\ge5$ be a rational prime, then $p\nmid i(K)$.
     \end{theorem}
 \begin{remark}
 	Note that when $a=0$ and $b$ is a non-zero  odd integer, then the index of  $K$ is $1.$
 \end{remark}
 \begin{corollary}
 	Let $K=\Q(\theta)$ be an algebraic number field generated by a root  $\theta$ of an irreducible polynomial $f(x)=x^7+ax+b\in \Z[x].$ If $a\equiv 15\mod 72$ and $b\equiv 12 \mod 72$, then in view of Theorem \ref{Th1.1}, \ref{Th1.2}  and \ref{Th1.3}, we have $i(K)=1$
 \end{corollary}

      We now provide some examples of non-monogenic number fields.
      \begin{example}
      Let $K=\Q(\theta)$ where $\theta$ is a root of $f(x)=x^7+17x+51$. Note that $f(x)$ satisfies Eisenstein's criterion with respect to $17$, hence it is irreducible over $\mathbb{Q}$. So, by Corollary \ref{cor2} $K$ is non-monogenic. 
        \end{example}
          
          \begin{example}
              Let $K=\Q(\theta)$ with $\theta$ satisfying $f(x)=x^7+7x+56$. It can be easily seen that $f(x)$ is $7$-Eisenstein. Hence, in view of Corollary \ref{cor1} $K$ is non-monogenic.
                \end{example}

 \section {Preliminary Results}
 Let $K=\Q(\theta)$ be an algebraic number field with $\theta$ a root of a monic irreducible polynomial $f(x)$ belonging to $\Z[x]$. In what follows, $\mathbb{O}_K$ will stand for the ring of algebraic integers of $K$.  For a rational prime $p$, let $\F_p$ denote the finite field with $p$ elements. The following  lemma (cf. \cite[Theorem 2.2]{hs}) will play an important role in the proof of Theorems \ref{Th1.1}, \ref{Th1.2} and \ref{Th1.3}. 
 \begin{lemma}\label{Th 1.12} Let $K$ be an algebraic number field and $p$ be a rational prime. Then  $p$ is a  common index divisor of $K$ if and only if for some positive integer $h$, the number of distinct  prime ideals of $\mathbb{O}_K$ lying above $p$ having residual degree $h$ is greater than the number of monic irreducible polynomials of degree $h$ in $\F_p[x]$. \end{lemma}
 We shall first introduce  the notion of Gauss valuation, $\phi$-Newton polygon and Newton polygon of second order, where $\phi(x)$  belonging to $\Z_p[x]$ is a monic polynomial with $\overline{\phi}(x)$ irreducible over $\F_p$.
 
 \begin{definition}\label{A} 
The Gauss valuation  of the field $\Q_p(x)$ of rational functions in an indeterminate $x$ which extends the valuation $v_p$ of $\Q_p$ and is defined on $\Q_p[x]$ by 
\begin{equation*}\label{Gau}
 v_{p,x}(\ a_0+a_1x+a_2x^2+.....+a_sx^s)= \displaystyle\min\{v_p(a_i),1\leq i\leq s \}, a_i\in \Q_p.
  \end{equation*}
\end{definition}
 \begin{definition}\label{p1.6}
Let $p$ be a rational  prime. Let $\phi(x)\in\Z_p[x]$ be a monic polynomial which is irreducible modulo $p$ and $f(x)\in\Z_p[x]$ be a monic polynomial not divisible by $\phi(x)$. Let $\displaystyle\sum_{i=0}^{n}a_i(x)\phi(x)^i$ with $\deg a_i(x)<\deg\phi(x)$, $a_n(x)\neq 0$ be the $\phi(x)$-expansion of $f(x)$ obtained on dividing it by the successive powers of $\phi(x)$. Let $P_i$ stand for the point in the plane having coordinates $(i,v_{p,x}(a_{n-i}(x)))$ when $a_{n-i}(x)\neq 0$, $0\leq i\leq n$. Let $\mu_{ij}$ denote the slope of the line joining the point $P_i$ with $P_j$ if $a_{n-i}(x)a_{n-j}(x)\neq 0$.  Let $i_1$ be the largest positive index not exceeding $n$ such that\\
 \centerline{$\mu_{0i_1}=\min\{~\mu_{0j}~|~0<j\leq n,~a_{n-j}(x)\neq0\}.$}
 
 \noindent
 If $i_1<n,$ let $i_2$ be the largest index such that $i_1<i_2\leq n$ with\\
 \centerline{$\mu_{i_1i_2}=\min\{~\mu_{i_1j}~|~i_1<j\leq n,~a_{n-j}(x)\neq0\}$} 
  and so on. The $\phi$-Newton polygon of $f(x)$ with respect to $p$ is the polygonal path having segments
 $P_{0}P_{i_1},P_{i_1}P_{i_2},\ldots,P_{i_{k-1}}P_{i_k}$ with $i_k=n$. These segments are called the edges of the $\phi$-Newton polygon and their slopes form a strictly increasing sequence; these slopes are non-negative as $f(x)$ is a monic polynomial with coefficients in $\Z_p$.
  \end{definition}
   \begin{definition}\label{p1.10}
   Let $\phi(x) \in\Z_p[x]$  be a monic polynomial which is irreducible modulo a rational prime $p$ having a root $\alpha$ in the algebraic closure $\widetilde{\Q}_{p}$ of $\Q_p$. Let $f(x) \in \Z_p[x]$ be a monic polynomial not divisible by $\phi(x)$ with $\phi(x)$-expansion $\phi(x)^n + a_{n-1}(x)\phi(x)^{n-1} + \cdots + a_0(x)$ such that $\overline{f}(x)$ is a power of $\overline{\phi}(x)$. Suppose that the $\phi$-Newton polygon of $f(x)$  consists of a single edge, say $S$, having positive slope denoted by $\frac{l}{e}$ with $l, e$ coprime, i.e.,  $$\min\bigg\{\frac{v_{p,x}(a_{n-i}(x))}{i}~\mid~1\leq i\leq n\bigg\} = \frac{v_{p,x}(a_0(x))}{n} = \frac{l}{e}$$ so that $n$ is divisible by $e$, say $n=et$ and $v_{p,x}(a_{n-ej}(x)) \geq lj$ for $1\leq j\leq t$. Thus the polynomial $b_j(x):=\frac{a_{n-ej}(x)}{p^{lj}}$   has coefficients in $\Z_p$ and hence $b_j(\alpha)\in \Z_p[\alpha]$ for $1\leq j \leq t$.  The polynomial $T(Y)$ in an indeterminate $Y$ defined by   $T(Y) = Y^t + \sum\limits_{j=1}^{t} \overline{b_j}(\overline{\alpha})Y^{t-j}$ having coefficients in $\F_p[\overline{\alpha}]\cong \frac{\F_p[x]}{\langle\phi(x)\rangle}$ is called residual polynomial of $f
   (x)$ with respect to $(\phi,S)$.
   \end{definition}
   The following definition gives the notion of residual polynomial when $f(x)$ is more general.
   \begin{definition}\label{p1.11} Let $\phi(x), \alpha$ be as in Definition \ref{p1.10}.  Let $g(x)\in \Z_p[x]$ be a monic polynomial not divisible by $\phi(x)$ such that $\overline{g}(x)$ is a power of $\overline{\phi}(x)$. Let $\lambda_1 < \cdots < \lambda_k$ be the slopes of the edges of the $\phi$-Newton polygon of $g(x)$ and $S_i$ denote the edge with slope $\lambda_i$. In view of a classical result proved by Ore (cf. \cite[Theorem 1.5]{CMS}, \cite[Theorem 1.1]{SKS1}), we can write $g(x) = g_1(x)\cdots g_k(x)$, where the $\phi$-Newton polygon of $g_i(x) \in \Z_{{p}}[x]$ has a single edge, say $S_i'$, which is a translate of $S_i$. Let $T_i(Y)$ belonging to ${\F}_{p}[\overline{\alpha}][Y]$ denote the residual polynomial of  $g_i(x)$ with respect to ($\phi,~S_i'$) described as in Definition \ref{p1.10}.  For convenience, the polynomial $T_i(Y)$  will be referred to as the residual  polynomial  of   $g(x)$ with respect to $(\phi,S_i)$. The polynomial $g(x)$ is said to be $p$-regular with respect to $\phi$ if none of the polynomials $T_i(Y)$  has a repeated root in the algebraic closure of $\F_p$, $1\leq i\leq k$. In general, if $F(x)$ belonging to $\Z_p[x]$ is a monic polynomial and $\overline{f}(x) = \overline{\phi}_{1}(x)^{e_1}\cdots\overline{\phi}_r{(x)}^{e_r}$ is its factorization modulo $p$ into irreducible polynomials with each $\phi_i(x)$ belonging to $\Z_p[x]$ monic and $e_i > 0$, then by Hensel's Lemma there exist monic polynomials $f_1(x), \cdots, f_r(x)$ belonging to $\Z_{{p}}[x]$ such that $f(x) = f_1(x)\cdots f_r(x)$ and $\overline{f}_i(x) = \overline{\phi}_i(x)^{e_i}$ for each $i$. The polynomial $f(x)$ is said to be $p$-regular (with respect to $\phi_1, \cdots, \phi_r$) if each $f_i(x)$ is ${p}$-regular with respect to $\phi_i$.
   \end{definition}
   To determine the number of distinct prime ideals of $\mathbb{O}_K$ lying above a rational prime $p$, we will use the Newton polygon of second order and the following theorem which is a weaker version of Theorem 1.2 of \cite{SKS1}.
   \begin{theorem}\label{Th 1}
   Let $L=\Q(\xi)$ be an algebraic number field with $\xi$ satisfying an  irreducible polynomial $g(x)\in \Z[x]$ and $p$ be a rational prime. Let $ \overline{\phi}_{1}(x)^{e_1}\cdots\overline{\phi}_r{(x)}^{e_r}$ be the  factorization of $g(x)$ modulo  $p$ into powers of distinct irreducible polynomials  over $\F_p$ with each $\phi_i(x)\neq g(x)$ belonging to $\Z[x]$  monic. Suppose that  the $\phi_i$-Newton polygon of $g(x)$ has $k_i$ edges, say $S_{ij}$ having slopes $\lambda_{ij}=\frac{l_{ij}}{e_{ij}} $  with $\gcd~(l_{ij},~e_{ij})=1$ for $1\leq j\leq k_i$. If  $T_{ij}(Y) = \prod\limits_{s=1}^{s_{ij}}U_{ijs}(Y)$ is the factorization of  the residual polynomial $T_{ij}(Y)$ into distinct irreducible factors over $\F_p$  with respect to $(\phi_i,~S_{ij})$  for $1\leq j\leq k_i$, then $$p\mathbb{O}_L=\displaystyle\prod_{i=1}^{r}\displaystyle\prod_{j=1}^{k_i}\displaystyle\prod_{s=1}^{s_{ij}}\mathfrak p_{ijs}^{e_{ij}},$$ where $\mathfrak p_{ijs}$ are distinct prime ideals of $\mathbb{O}_L$ having residual degree $\deg \phi_i(x)\times\deg U_{ijs}(Y).$
   \end{theorem}
   
  Let $L=\Q(\gamma)$ where $\gamma$ is a root of a monic polynomial $g(x)=a_nx^n+\cdots+a_0\in\Z[x],~a_0\neq 0$. Let $p$ be a prime number such that  $g(x)\equiv x^n(p)$. Suppose that the $p$-Newton polygon of $g(x)$ consists of a single edge with positive slope $\lambda=\frac{l}{e},$ where $\gcd(l,~e)=1$. Let the residual polynomial $T_g(Y)\in\F_p[Y]$ of $g(x)$ is a power of monic irreducible polynomial $\psi(Y)$ over $\F_p$, i.e., $T_g(Y)=\psi(Y)^s$ in $\F_p[Y]$, where $s\geq 2$. In this case, we construct a key polynomial $\Phi(x)$ attached with the slope $\lambda$ such that the following hold: \\(i) $\Phi(x)$ is congruent to a power of $x$ modulo $p$.\\
  (ii) The $p$-Newton polygon of $\Phi(x)$ of first order is one-sided with slope $\lambda$.\\
  (iii) The residual polynomial of $\Phi(x)$  with respect to $p$ is $\psi(Y)$ in $\F_p[Y]$.\\
  (iv) $\deg \Phi(x)= e\deg \psi(Y)$. \\
  As described in \cite[Section 2.2]{gaa}, the data $(x;~\lambda,~\psi(Y))$ determines a $p $-adic valuation $V$ of the field $\Q_p(x)$ which satisfies the following properties:\\
  (i) $V(x)=l$ where $\lambda=\frac{l}{e}~$ with $\gcd(l,~e)=1$.\\
  (ii) If $p(x)=\displaystyle\sum_{0\leq i}^{}b_ix^i\in\Z_p[x]$ is any polynomial, then \begin{equation}\label{eq:an}
  V(p(x))=e \displaystyle\min_{0\leq i}^{}\{v_p(b_i)+i \lambda\}
  \end{equation}
 We define the above valuation $V$ to  the valuation of second order.\\
\indent  If $g(x)=\displaystyle\sum_{i=0}^{u} a_i(x)\Phi(x)^i\in\Z_p[x]$ is a $\Phi$-adic expansion of $g(x)$, then the Newton polygon of $g(x)$ with respect to $V$ (also called $V$-Newton polygon of $g(x) $ of second order) is the lower convex hull of the set of the points $(i, V(a_{u-i}(x)\Phi(x)^{i}))$ of the  Euclidean plane.\\
 Let the $V$-Newton polygon of $g(x)$ of second order has $k$-sides $E_1,\cdots , E_k$ with positive slopes $\lambda_1,\cdots,\lambda_k$.  Let $\lambda_t=\frac{l_t}{e_t}$ with 
 $\gcd(l_t,~e_t)=1$ and $[a_t, b_t]$ denote the projection
 to the horizontal axis of the side of slope $\lambda_t$ for $1\leq t \leq k.$ Then, there is a natural residual polynomial $\psi_t(Y)$ of second order attached to each side $E_t,$ whose degree coincides
 with the degree of the side (i.e. $\frac{b_t-a_t}{e_t}$) \cite[Section 2.5]{gaa}. Only those integral points of the
$ V $-Newton polygon of $g(x)$ which lie on the side, determine a non-zero coefficient of this
 second order residual polynomial.
 We define $g(x)$ to be $\psi_t $-regular when the second order residual polynomial $\psi_t(Y)$ attached to the side $E_t$ of the $V$-Newton polygon of $g(x)$ of second order is separable in
 $\frac{\F_p[Y ]}{\langle\psi(Y)\rangle}$. We define $g(x)$ to be $V$-regular if $g(x)$ is $\psi_t$-regular for each $t$, $1\leq t\leq k$. Further, if each residual polynomial $\psi_t(Y)$, $t\in\{1,~2,\cdots,k\}$, is irreducible in
 $\frac{\F_p[Y ]}{\langle\psi(Y)\rangle}$ , then each $\psi_t(Y)$ provides a prime ideal having residual degree  $\deg \psi\cdot \deg \psi_t$ and ramification index  $e\cdot e_t $. 
   \section{Proof of Theorems \ref{Th1.1}, \ref{Th1.2} and \ref{Th1.3}.}
   \begin{proof}[Proof of Theorem \ref{Th1.1}] 
   	 If $2$ is a common index divisor of $K$, then $2\mid D$, i.e. $2\mid b$.\\
   	\textbf{Case A1:} $a\equiv1\mod2$,  $b\equiv0\mod2$. In this case, $f(x)\equiv x(1+x+x^2)^2(x+1)^2\mod 2$. Let $\phi_1(x)=x$, $\phi_2(x)=1+x+x^2$ and  $\phi_3(x)=x+1$. The $\phi_1$-Newton polygon of $f(x)$ has a single edge joining the points $(0,~0)$ and $(7,~v_2(b))$. The residual polynomial associated with this edge is linear. Therefore $\phi$ provides one prime ideal say $\mathfrak{ p}_1$ of residual degree $1$. So 
   	\begin{equation}\label{eq3.1}
   		2\mathbb{O}_K=\mathfrak{ p}_1\mathfrak{P},\hspace{1.5cm}~\text{where}~ \mathfrak{P} ~\text{is an ideal of}~ \mathbb{O}_K
   	\end{equation}
   Using Theorem \ref{Th 1}, we see that $\phi_2$ provides prime ideals of residual degree multiple of $2.$ Therefore
    keeping in mind Lemma \ref{Th 1.12}, $2\mid i(K)$ if and only if $\phi_2$ and $\phi_3$ provides atleast two prime ideals of residual degree $t$, where $t\in \{1,2\}$. The $\phi_i$, for $i=2,3$  expansion of $f(x)$ is 
\begin{equation}
	(x-3)\phi_2^3+(3x+5)\phi_2^2-(4x+2)\phi_2+b+x(a+1)
\end{equation}
\begin{equation}
	\phi_3^7-7\phi_3^6+21\phi_3^5-35\phi_3^4+35\phi_3^3-21\phi_3^2+(a+7)\phi_3+(b-a-1)
\end{equation}
The $\phi_2$-Newton polygon of $f(x)$ is the lower convex hull of the points $(0,~0)$, $(1,~0)$, $(2,~1)$ and $(3,\min\{v_2(b), v_2(a+1)\})$ and $\phi_3$-Newton polygon of $f(x)$ is the lower convex hull of the points $(0,~0)$, $(1,~0)$, $(2,~0)$, $(3,~0)$ , $(4,~0)$, $(5,~0)$, $(6,v_2(a+7))$ and $(7,~v_2(b-a-1))$.
\indent Let $a\equiv 3\mod4$ and $b\equiv 2\mod4$, then $v_2(a+1)\ge 2$. For each $i=2,3$, $\phi_i$-Newton polygon of $f(x)$ has a single edge of slope $\frac{1}{2}$. The residual polynomial attached to this  edge is linear. Thus  $\phi_2$ provides one prime ideal say $\mathfrak{ p}_2$ of residual degree $2$ and $\phi_3$ provides one prime ideal say $\mathfrak{ p}_3$ of residual degree 1. So by Theorem {\ref {Th 1}}, $\mathfrak{P}=\mathfrak{ p}_2^2\mathfrak{ p}_3^2$. Thus $2\nmid i(K).$\\
\indent Let $a\equiv 3\mod 8$ and  $b\equiv 0\mod 8$ , then for each $i=2,3$, the $\phi_i$-Newton polygon of $f(x)$ has a single edge of positive slope. The residual polynomial of $f(x)$ associated to this edge with respect to $\phi_2$ is $(Y-(x+1))(Y-1)$ over  $\frac{\F_2[x]}{\langle\phi_2(x)\rangle}$ and with respect to $\phi_3$ is $Y^2+Y+\bar{1}\in F_2[Y].$ Therefore $\mathfrak{P}=\mathfrak{ p}_2\mathfrak{ p}_3\mathfrak{ p}_4$, where residual degree of each $\mathfrak{ p}_i$, for $i=2,3,4$  is $2$. Thus $2\mid i(K)$.\\
\indent Let $a\equiv 7\mod 8$ and  $b\equiv 0\mod 8$. Then for each $i=2,3$, the  $\phi_i$-Newton polygon of $f(x)$ has a two edges of positive slope. The residual polynomial of $f(x)$ associated to each edge is linear.  Therefore $\mathfrak{P}=\mathfrak{ p}_2\mathfrak{ p}_3\mathfrak{ p}_4\mathfrak{ p}_5$, where residual degree of each $\mathfrak{ p}_2$, $\mathfrak{ p}_3$ is 2 and of $\mathfrak{ p}_4$, $\mathfrak{ p}_5$ is $1$. Thus $2\mid i(K)$.\\
\indent Let $a\equiv 3\mod 8$ and  $b\equiv 4\mod 8$.
 Then the $\phi_2$-Newton polygon of $f(x)$ has a single edge of positive slope. The residual polynomial of $f(x)$ associated to this edge with respect to $\phi_2$ is $Y^2+xY+1$ over  $\frac{\F_2[x]}{\langle\phi_2(x)\rangle}$. The  $\phi_3$-Newton polygon of $f(x)$ has  two edges of positive slope. The residual polynomial of $f(x)$ associated to each edge is linear. So $\mathfrak{P}=\mathfrak{ p}_2\mathfrak{ p}_3\mathfrak{ p}_4$, where residual degree of $\mathfrak{ p}_2$ is $4$ and of $\mathfrak{ p}_3$, $\mathfrak{ p}_4$ is $1$. Hence $2\mid i(K)$.\\
 \indent Let $a\equiv 7\mod 8$ and  $b\equiv 4\mod 8$.
 Then the $\phi_2$-Newton polygon of $f(x)$ has a single edge of positive slope. The residual polynomial of $f(x)$ associated to this edge with respect to $\phi_2$ is $Y^2+xY+x$ over  $\frac{\F_2[x]}{\langle\phi_2(x)\rangle}$. The  $\phi_3$-Newton polygon of $f(x)$ has  a single  edge of positive slope. The residual polynomial of $f(x)$ associated to the edge is $Y^2+Y+\bar{1}$. So $\mathfrak{P}=\mathfrak{ p}_2\mathfrak{ p}_3$, where residual degree of $\mathfrak{ p}_2$ and  $\mathfrak{ p}_3$ is $4$ and $2$ respectively. Hence $2\nmid i(K)$.\\
\indent Let $a\equiv 1\mod4$ and $b\equiv 0\mod 2$. Then $\phi_2$-Newton polygon of $f(x)$ has a single edge of slope $\frac{1}{2}$. Thus $\phi_2$ provide one prime ideal say $\mathfrak{ p}_2$ of residual degree $1$. Therefore by using \eqref{eq3.1}, we have
\begin{equation}
	2\mathbb{O}_K=\mathfrak{p}_1\mathfrak{ p}_2^2\mathfrak{I}, \hspace{1cm}\text{where}~ \mathfrak{I}~\text{ is an ideal of}~\mathbb{O}_K.
\end{equation}  It is clear that $2\mid i(K)$ if and only if $\phi_3$ provides either  two prime ideals of residual degree $2$ each or atleast one prime ideal of residual degree $1$.\\ If $a\equiv 1\mod 4$ and  $b\equiv 0\mod4$, then  $\phi_3$  provides one prime ideal say $\mathfrak{ p}_3$  of residual degree $1$. Thus $\mathfrak{I}=\mathfrak{ p}_3^2$. So, $2\mid i(K).$\\
 \indent One can observe that if  $a\equiv 1\mod 4$, $b\equiv 2\mod 4$, $v_2(b-a-1)<2v_2(a+7)$ and $v_2(b-a-1)$ is even, then the $\phi_3$-Newton polygon has a single edge of positive slope whose residual polynomial  is not square-free, therefore we find some rational $\mu$ such that $f(x)$ is $x-\mu$ regular. \\
 \indent Let  $a\equiv 1\mod4$ and $b\equiv 2\mod4$ and  take $\mu=\frac{-7b}{6a}$, then $v_2(\mu)=0$. Let  $\phi_3(x)=x-\mu$, then the $\phi_3$ expansion of $f(x)$ is 
 \begin{equation}
 	\phi_3^7+7\mu\phi_3^6+21\mu^2\phi_3^5+35\mu^3\phi_3^4+35\mu^4\phi_3^3+21\mu^5\phi_3^2+f'(\mu)\phi_3+f(\mu)
 \end{equation}
A simple calculation shows that $f(\mu)=\frac{-Db}{6^7a^7}$ and $f'(\mu)=\frac{D}{6^6a^6}$. Clearly $v_2(f(\mu))=v_2(f'(\mu))=v_2(D)-6$. It is easy to verify that $v_2(D)-6\ge1$. If $v_2(D)$ is odd, then $\phi_3$-Newton polygon of $f(x)$ has one edge of positive slope joining the points $(5,~0)$ and $(7, ~v_2(D)-6)$. So $\phi_3$ provides one prime ideal say $\mathfrak{ p}_3$ of residual degree $1$. Therefore $\mathfrak{I}=\mathfrak{ p}_3^2$ and so $2\mid i(K)$.\\ If $v_2(D)$ is even, then $f(x)$ is not $x-\mu$ regular. So we choose another rational number.\\
\indent Here we  have $a\equiv 1\mod4$, $b\equiv 2\mod4$ and $v_2(D)$ is even. Take $u=\frac{v_2(D)-6}{2}$  and define $\rho=2^u-\mu$, then $v_2(\rho)=0$. Let $\phi_3(x)=x-\rho$. The $\phi_3$ expansion of $f(x)$ is 
\begin{equation}
 	\phi_3^7+7\rho\phi_3^6+21\rho^2\phi_3^5+35\rho^3\phi_3^4+35\rho^4\phi_3^3+21\rho^5\phi_3^2+f'(\rho)\phi_3+f(\rho)
\end{equation}
One can verify that $v_2(f(\rho))\ge 2u+1$ and $v_2(f'(\rho))= u+1$. If $v_2(f(\rho))= 2u+1$, then $\phi_3$ provides one prime ideal say $\mathfrak{ p}_3$ of residual degree $1$. Therefore $\mathfrak{I}=\mathfrak{ p}_3^2$. So $2\mid i(K)$. If $v_2(f(\rho))= 2u+2$, then $\phi_3$ provides one prime ideal say $\mathfrak{ p}_3$ of residual degree $2$. Therefore $\mathfrak{I}=\mathfrak{ p}_3$. So $2\nmid i(K)$. If $v_2(f(\rho))\ge 2u+3$, then $\phi_3$ provides two prime ideal say $\mathfrak{ p}_3$  and $\mathfrak{ p}_4$ of residual degree $1$ each. Therefore $\mathfrak{I}=\mathfrak{ p}_3\mathfrak{ p}_4$. So $2\mid i(K)$\\\\
   \noindent\textbf{Case A2:} $a\equiv0\mod2$, $b\equiv0\mod2$. Here $f(x)\equiv x^7\mod 2$. The $x$-Newton polygon of $f(x)$ is the lower convex hull of the points $(0,~0)$, $(6,~v_2(a))$ and $(7,~v_2(b))$.\\ 
   \indent  Let $7v_2(a)>6v_2(b)$, then by  \eqref{1}, $v_2(b)<7$. The $x$-Newton polygon of $f(x)$ has a single edge of  positive slope $\frac{v_2(b)}{7}$. The residual polynomial associated to this edge is linear.  Therefore $2\mathbb{O}_K=\mathfrak{p}^7$. Thus $2\nmid i(K)$.\\
\indent For the remaining case, let  $7v_2(a)<6v_2(b)$. Then by  \eqref{1}, we have $1\le v_2(a)\le 5$. The $x$-Newton polygon of $f(x)$ has  two  edges of positive slope. The first edge say $S_1$ is line segment joining the points $(0,~0)$ and $(6,~v_2(a))$ with slope $\frac{v_2(a)}{6}$. The second edge say $S_2$ is the line segment joining the points $(6,~v_2(a))$ and $(7,~v_2(b))$. The residual polynomial associated to the edge $S_2$ is linear. Therefore $S_2$ provides one prime ideal say $\mathfrak{ q}$  of residual degree $1$. Thus
\begin{equation}\label{eq3.7}
	2\mathbb{O}_K=\mathfrak{R}\mathfrak{ q}, \hspace{1cm}~\text{where} ~\mathfrak{R}~\text{is an ideal of }~ \mathbb{O}_K 
\end{equation} 
  If $v_2(a)\in \{1,5\}$, then $S_1$  provides one prime ideal say $\mathfrak{ p}_1$ of residual degree $1$. So $\mathfrak{R}=\mathfrak{ p}_1^6$. Hence $2\nmid i(K)$.\\ If $v_2(a)=3$, then the residual polynomial associated to $S_1$  is $(Y+\bar{1})(Y^2+Y+\bar{1})$. Therefore $\mathfrak{R}=\mathfrak{ p}_1^2\mathfrak{ p}_2^3$, where residual degree of $\mathfrak{ p}_1$ and $\mathfrak{ p}_2$ is $1$ and $2$ respectively. Thus $2\nmid i(K)$.\\
If $v_2(a)=2$, then $\lambda=\frac{1}{3}$ and the residual polynomial associated to $S_1$ is $(Y+\bar{1})^2$, i.e. residual polynomial is not square-free.  Let $\psi(Y)=Y+\bar{1}$, $h=1$ and $e=3$. Since $e>1$, for the prime ideal factorization of $2\mathbb{O}_K$, we shall use higher order Newton polygons. Take $\Phi(x)=x^3+2$, then $\Phi(x)$ is the key polynomial attached to the data $(x, \lambda, \psi(Y))$. As in \eqref{eq:an}, we can define valuation $V$ of second order such that $V(\Phi)=3$, $V(x)=1$ and $V(2)=3$. The $\Phi(x)$ expansion of $f(x)$ is 
\begin{equation}
	f(x)=x\Phi^2(x)-4x\Phi(x)+(a+4)x+b
\end{equation}

\noindent The $\Phi$-Newton polygon of $f(x)$ of second order is the lower convex hull of the points $(0,~7)$, $(1,~10)$ and $(2,~v)$, where $v=\min\{3v_2(b), 3v_2(a+4)+1\}$. Since $v_2(a)=2$, we have $v_2(b)\ge3$.\\ 
If  $a\equiv 4\mod16$ and $b\equiv 0\mod 16$, then $v=10$. The $\Phi$-Newton polygon of $f(x)$ of second order has a single edge and the residual polynomial attached to this edge is linear. So, $\mathfrak{R}=\mathfrak{ p}_1^6$, where residual degree of $\mathfrak{ p}_1$ is $1$. Thus $2\nmid i(K)$.\\
If  $a\equiv 12\mod16$ and $b\equiv 16\mod 32$, then $v=12$. The $\Phi$-Newton polygon of $f(x)$ of second order has a single edge whose  residual polynomial  is linear. So, $\mathfrak{R}=\mathfrak{ p}_1^6$, where residual degree of $\mathfrak{ p}_1$ is $1$. Thus $2\nmid i(K)$\\
 If  $a\equiv 12\mod32$ and $b\equiv 0\mod 32$, Then $v=13$, then $\Phi$-Newton polygon of $f(x)$ of second order has a single edge and residual polynomial attached to this edge is $Y^2+Y+\bar{1}$. So, $\mathfrak{R}=\mathfrak{ p}_2^3$,where residual degree of $\mathfrak{ p}_2$ is $2$. Thus $2\nmid i(K)$.\\
  If  $a\equiv 28\mod32$ and $b\equiv 0\mod 32$, then $v>13$, then $\Phi$-Newton polygon of $f(x)$ of second order has two  edges of positive slope. The residual polynomial attached to each edge is linear. So, $\Phi$ provides two prime ideals say $\mathfrak{ p}_1$ and $\mathfrak{ p}_2$ of residual degree $1$ each. Therefore $\mathfrak{R}=\mathfrak{ p}_1^3\mathfrak{ p}_2^3$. hence $2\mid i(K)$.\\
\indent Note that if $v_2(b)=3$, then the $\Phi$-Newton polygon of $f(x)$ of second order has a single edge. The residual polynomial attached to this edge is not square-free. So we shall use another key polynomial.\\
\indent Let $v_2(b)=3$ and take $\Phi(x)=x^3+2x+2$. Then $\Phi(x)$ is key polynomial attached to $(x,\frac{1}{3},\psi(Y))$. Let $V$ be same as given above. The $\Phi(x)$ expansion of $f(x)$ is 
\begin{equation}
	f(x)=x\Phi^2(x)+(4-4x-4x^2)\Phi(x)+(b-8+(a-4)x+8x^2)
\end{equation}
 The $\Phi$-Newton polygon of $f(x)$ of second order being the lower convex hull of the points $(0,~7)$, $(1,~9)$ and $(2,~w)$, where $w=V(b-8+(a-4)x+8x^2)$.\\
  If $a\equiv 12\mod16$ and $b\equiv 8\mod 16$, then $w=10$. The $\Phi$ provide one prime ideal say $\mathfrak{ p}_1$ of residual degree $1$. So, $\mathfrak{R}=\mathfrak{p}_1^6$. Thus $2\nmid i(K)$.\\ If $a\equiv 4\mod16$ and $b\equiv 8\mod 16$, then $w=11$, then $\Phi$ provides one prime ideal say $\mathfrak{ p}_1$ of residual degree $2$. So, $\mathfrak{R}=\mathfrak{ p}_1^3$. Thus $2\nmid i(K)$.\\
\indent  Let $v_2(a)=4$, then $v_2(b)\ge5$, $\lambda=\frac{2}{3}$, $h=2$ and $e=3$. Take $\Phi(x)=x^3+4$. The valuation $V$ of second order defined in \eqref{eq:an} is such that $V(\Phi)=6$, $V(x)=2$ and $V(2)=3$. The $\Phi(x)$ expansion of $f(x)$ is 
 \begin{equation}
      x\Phi^2(x)-8x\Phi(x)+(a+16)x+b
 \end{equation}
Let $w'= \min \{3v_2(b), 3v_2(a+16)+2\})$, then $w'\ge 15$. The $\Phi$-Newton polygon of $f(x)$ of second order is the lower convex hull of the points $(0,~14)$, $(1,~17)$ and $(2,~w')$.\\
 If $a\equiv 16\mod32$ and $b\equiv 32\mod 64$, then $w'=15$. The $\Phi$ provide one prime ideal say $\mathfrak{p}_1$ of residual degree $1$. So, $\mathfrak{R}=\mathfrak{ p}_1^6$. Therefore $2\nmid i(K)$.\\
 If $a\equiv 16\mod64$ and $b\equiv 0\mod 128$, then $w'=17$. So, the $\Phi$-Newton polygon of $f(x)$ of second order has a single edge of positive slope. The residual polynomial attached to this edge is linear. Therefore $\mathfrak{R}=\mathfrak{ p}_1^6$. Thus $2\nmid i(K)$.\\
 If $a\equiv 48\mod64$ and $b\equiv 0\mod 128$, then $w'=20$. So, the $\Phi$-Newton polygon of $f(x)$ of second order has a single edge of positive slope. The residual polynomial attached to this edge is $Y^2+Y+\bar{1}$. Therefore $\mathfrak{R}=\mathfrak{ p}_1^3$. Thus $2\nmid i(K)$.\\
   If $a\equiv 16\mod64$ and $b\equiv 64\mod 128$, then $w'=17$. Hence $\mathfrak{R}=\mathfrak{ p}_1^6$. Thus $2\nmid i(K)$.\\
    If $a\equiv 48\mod64$ and $b\equiv 64\mod 128$, then $w'=18$. So,  Take $\Phi(x)=x^3+4x+4$. The $\Phi$ expansion of $f(x)$ is 
\begin{equation}
	x\Phi^2(x)-(8x+4x^2)\Phi(x)+b-64+x(a-48)+32x^2
\end{equation}
The $\Phi$-Newton polygon of $f(x)$ of second order is the lower convex hull of the points $(0,~14)$, $(1,~16)$ and $(2,~19)$. Thus $\Phi$ provide two prime ideals say $\mathfrak{ p}_1$ and $\mathfrak{ p}_2$  of residual degree $1$ each. So $\mathfrak{R}= \mathfrak{p}_1^3\mathfrak{p}_2^3$. Therefore $2\mid i(K)$.                       
  This completes the proof of the theorem. \end{proof}

\begin{proof}[Proof of Theorem \ref{Th1.2}]  If $3\mid i(K)$, then by using $\eqref{eq1}$ and  $\eqref{eq2}$, we have $3\mid b$.\\
\textbf{Case B1:} Let $a\equiv 0\mod3$ and $b\equiv 0\mod 3$, then $f(x)\equiv x^7\mod 3$. The $x$-Newton polygon of $f(x)$ is the lower convex hull of the points $(0,~0)$, $(6,~v_3(a))$ and $(7,~v_3(b))$. If $7v_3(a)>6v_3(b)$, then $x$ provides one prime ideal of residual degree $1$. So $3\mathbb{O}_K=\mathfrak{p}^6$.\\
Let $7v_3(a)<6v_3(b)$. Then the $x$-Newton polygon of $f(x)$ being the lower convex hull of the points $(0,~0)$, $(6,~v_3(a))$ and $(7,~v_3(b))$ has two edges of positive slopes. The first say $S_1$ is the line segment joining the points $(0,~0)$ and $(6,~v_3(a))$. The second edge say $S_2$ is the line segment joining the points $(6,~v_3(a))$ and $(7,~v_3(b))$. The residual polynomial associated with the edge $S_2$ is linear. Thus $3\mathbb{O}_K=\mathfrak{ L}\mathfrak{ q}$, for some ideal $\mathfrak{L}$ of  $\mathbb{O}_K$. Using \eqref{1} and $7v_3(a)>6v_3(b)$, we have $v_3(b)\in\{1,2,3,4,5\}.$\\ If $v_3(a)\in \{1,5\}$, then $S_1$ provides
one prime ideal say $\mathfrak{ p}_1$ of residual degree $1$. So  $\mathfrak{L}=\mathfrak{ p}_1^6$. \\ If $v_3(a)\in \{2,4\}$, then the residual polynomial associated with $S_1$ is $(Y-\bar{1})(Y+\bar{1})$. So  $\mathfrak{L}=\mathfrak{ p}_1^3\mathfrak{ p}_2^3$.\\
Let $v_2(a)=3$, then the residual polynimial associated with $S_1$  is $(Y+\delta)^2$, where $\delta\in\{1,-1\}$. Let $\psi(Y)=Y+\delta$, $\lambda=\frac{1}{2}$. Take $\phi(x)=x^2+3\delta$, $h=1$ and $e=2$, then $\Phi(x)$ is key  polynomial attached to the data $(x,\lambda,\psi(Y))$. The valuation $V$ on $\mathbb{Q}_2[x]$, given in \eqref{eq:an} is such that $V(\Phi)=2$, $V(x)=1$ and $V(3)=2$. The $\Phi(x)$ expansion of $f(x)$ is 
\begin{equation}
	x\Phi^3(x)-9\delta\Phi^2(x)+27x\Phi(x)+b+x(a-27\delta)
\end{equation}
The $\Phi(x)$-Newton polygon of second order is the lower convex hull of the points $(0,~7)$, $(1,~9)$, $(2,~9)$ and  $(3,~\min\{2v_3(b),2v_3(a-27\delta)+1\}) $. As $v_3(a)=3$, so by using $\eqref{1}$, we have  $v_3(b)\ge4.$\\ 
If $v_3(b)=4$, then $\Phi$-Newton polygon of $f(x)$ has a single edge of slope $\frac{1}{3}$. The residual polynomial associated to this edge is linear. Therefore $\mathfrak{ L}=\mathfrak{ p}_1^6$, where degree of $\mathfrak{ p}_1$ is $1$.\\
If $v_3(b)>4$ and $v_3(a-27\delta)=4$, then $\Phi$-Newton polygon of $f(x)$ has a single edge of slope $\frac{2}{3}$. The residual polynomial associated to this edge is linear. Therefore $\mathfrak{ L}=\mathfrak{ p}_1^6$, where degree of $\mathfrak{ p}_1$ is $1$.\\
If $v_3(b)=5$ and $v_3(a-27\delta)>4$, then $\Phi$-Newton polygon of $f(x)$ has a single edge joining the points $(0,~7)$ and $(3,~10)$ with a lattice point $(2,~9)$ on it. The residual polynomial associated to this edge is $Y^3\pm Y\pm\bar{1}$, which is separable. Therefore $\mathfrak{ L}=\mathfrak{ p}_1^2$, where degree of $\mathfrak{ p}_1$ is $3$.\\
If $v_3(b)>5$ and $v_3(a-27\delta)>4$, then $\Phi$-Newton polygon of $f(x)$ has two  edges of positive slopes. The residual polynomial associated to the first edge is  $Y^2+\bar{1}$. Therefore $\mathfrak{ L}=\mathfrak{ p}_1^2\mathfrak{ p}_2^2$. Hence we conclude that when $a\equiv 0\mod 3$ and $b\equiv 0\mod3$, $3\nmid i(K)$.\\\\
\noindent\textbf{Case B2:} $a\equiv -1\mod 3$ and $b\equiv 0\mod 3$. In this case $f(x)\equiv x(x-1)^3(x+1)^3\mod 3$. Clearly $x$-Newton polygon of $f(x)$ provides one prime ideal say $\mathfrak{ q}$ of residual degree $1$. Therefore 
\begin{equation}
	3\mathbb{O}_K=\mathfrak{ q}\mathfrak{J} ,\hspace{1.5cm}~\text{where}~\mathfrak{J}~\text{is an ideal of}~\mathbb{O}_K
\end{equation}
Keepind in mind Lemma  \ref{Th 1.12}, $3\mid i(K)$ if and only if $x-1$ and $x+1$ provides atleast three prime ideals of residual degree $1$ each. Let $\phi_\delta(x)=(x+\delta)$, where $\delta\in \{-1,1\}$, then the $\phi_\delta(x)$  expansion of $f(x)$ is 
\begin{equation}
	\phi^7_\delta(x)-7\delta\phi^6_\delta(x)+21\phi^5_\delta(x)-35\delta\phi^4_\delta(x)+35\delta\phi^3_\delta(x)-21\delta\phi^2_\delta(x)+(a+7)\phi_\delta(x)+(b-\delta(a+1))
\end{equation}
The $\phi_\delta(x)$-Newton polygon of $f(x)$ is the lower convex hull of the points $(0,~0)$, $(1,~0)$, $(2,~0)$, $(3,~0)$, $(4,~0)$, $(5,~1)$, $(6,~v_3(a+7))$ and  $(7,~v_3(b-\delta(a+1))$.
If $v_3(a+1)=1$ and $v_3(b)>1$, then for each $\delta\in\{-1,1\}$, $\phi_\delta$ provides one prime ideal of residual degree $1$. Therefore $\mathfrak{ J}=\mathfrak{ p}_1^3\mathfrak{ p}_2^3$. Hence $3\nmid i(K)$.\\
 Note that when $v_3(a+1)=1$ and $v_3(b)=1$, then either $a\equiv 2\mod 9$ or $a\equiv 5\mod 9.$\\
 \indent Let $a\equiv 2\mod 9$, $b_3\equiv 1\mod 3$ and $v_3(b-a-1)=2$, then for each $\delta\in\{1,-1\}$, $\phi_\delta$-Newton polygon of $f(x)$ has a single edge of positive slope and the residual polynomial attached to this edge is linear. Thus $\mathfrak{ J}=\mathfrak{ p}_1^3\mathfrak{ p}_2^3$, where residual degree of $\mathfrak{p_1}$ and $\mathfrak{p_2}$ is $1$. Hence $3\nmid i(K)$.\\
\indent Let $a\equiv 2\mod 9$, $b_3\equiv 1\mod 3$, $v_3(a+7)=2$ and $v_3(b-a-1)\ge 3$. Then $\phi_{-1}$ provides one prime ideal of degree $1.$ The $\phi_{1}$-Newton polygon of $f(x)$ has two edges of positive slope. The first edge is the line segment joining the points $(4,~0)$ and $(6,~2)$. The second edge is the line segment joining the points $(6,~2)$ and $(7,~v_3(b-a-1))$. The residual polynomial attached to each edge is linear. Therefore $\mathfrak{ J}=\mathfrak{ p}_1^3\mathfrak{ p}_2^2\mathfrak{p_3}$, where for each $i=1,2,3$ the residual degree of $\mathfrak{p_i}$  is $1$. Hence $3\mid i(K)$.\\
\indent Let $a\equiv 2\mod 9$, $b_3\equiv 1\mod 3$, $v_3(a+7)\ge3$, $v_3(b-a-1)\ge 3$ and $2v_3(a+7)>v_3(b-a-1)+1$. Then $\phi_{-1}$ provides one prime ideal of residual  degree $1$ and $\phi_{1}$-Newton polygon of $f(x)$ has two edges of positive slope. The first edge is the line segment joinig the points $(4,~0)$ and $(5,~1)$ and the second edge join $(5,~1)$ with $(7,~v_3(b-a-1))$. The residual polynomial attached to the first edge is linear and to second edge is $(Y-\bar{1})(Y+\bar{1})$ or $Y^2+\bar{1}$ (according as $(b-a-1)_3\equiv -1\mod 3$ or $(b-a-1)_3\equiv 1\mod 3$). Thus either $\mathfrak{ J}=\mathfrak{ p}_1^3\mathfrak{ p}_2^2\mathfrak{p_3}$, where  the residual degree of $\mathfrak{p_i}=1$, for each $i=1,2,3$ or  $\mathfrak{ J}=\mathfrak{ p}_1^3\mathfrak{ p}_2^2\mathfrak{p_3}$, where  the residual degree of $\mathfrak{p_1}$, $\mathfrak{p_2}$   is $1$ and of $\mathfrak{p_3}$ is $2$.\\
\indent Let $a\equiv 2\mod 9$, $b_3\equiv 1\mod 3$, $v_3(a+7)\ge3$, $v_3(b-a-1)\ge 3$ and $2v_3(a+7)=v_3(b-a-1)+1$. Then $\phi_{-1}$ provides one prime ideal of residual  degree $1$ and $\phi_{1}$ has two edges of positive slope. The residual polynomial attached to the first edge is linear and to the second edge is $Y^2-Y-1$ or $Y^2+Y-1$. Thus  $\mathfrak{ J}=\mathfrak{ p}_1^3\mathfrak{ p}_2^2\mathfrak{p_3}$, where  the residual degree of $\mathfrak{p_1}$, $\mathfrak{p_2}$   is $1$ and of $\mathfrak{p_3}$ is $2$. So, $3\nmid i(K).$\\
\indent Let $a\equiv 2\mod 9$, $b_3\equiv 1\mod 3$, $v_3(a+7)\ge3$, $v_3(b-a-1)\ge 3$ and $2v_3(a+7)<v_3(b-a-1)+1$. Then $\phi_{-1}$ provides one prime ideal of residual  degree $1$ and $\phi_{1}$-Newton polygon of $f(x)$ has three edges of positive slope. The first edge is the line segment joining the points $(4,~0)$ and $(5,~1)$, the second edge join $(5,~1)$ with $(6,~v_3(a+7))$ and the third edge joins $(6,~v_3(a+7))$ and $(7,~v_3(b-a-1))$. The residual polynomial attached to the each edge is linear. Hence $\mathfrak{ J}=\mathfrak{ p}_1^3\mathfrak{ p}_2\mathfrak{p_3}\mathfrak{p_4}$, where  the residual degree of $\mathfrak{p_i}=1$, for each $i=1,2,3,4$  is $1$. Thus $3\mid i(K)$.\\
\indent Let $a\equiv 2\mod 9$, $b_3\equiv -1\mod 3$ and $v_3(b+a+1)=2$, then for each $\delta\in\{1,-1\}$, $\phi_\delta$-Newton polygon of $f(x)$ has a single edge of positive slope and the residual polynomial attached to this edge is linear. Thus $\mathfrak{ J}=\mathfrak{ p}_1^3\mathfrak{ p}_2^3$, where residual degree of $\mathfrak{p_1}$ and $\mathfrak{p_2}$ is $1$. Hence $3\nmid i(K)$.\\
\indent Let $a\equiv 2\mod 9$, $b_3\equiv -1\mod 3$, $v_3(a+7)=2$ and $v_3(b+a+1)\ge 3$. Then $\phi_{1}$ provides one prime ideal of degree $1.$ The $\phi_{-1}$-Newton polygon of $f(x)$ has two edges of positive slope. The first edge is the line segment joining the points $(4,~0)$ and $(6,~2)$. The second edge is the line segment joining the points $(6,~2)$ and $(7,~v_3(b-a-1))$. The residual polynomial attached to each edge is linear. Therefore $\mathfrak{ J}=\mathfrak{ p}_1^3\mathfrak{ p}_2^2\mathfrak{p_3}$, where for each $i=1,2,3$ the residual degree of $\mathfrak{p_i}$  is $1$. Hence $3\mid i(K)$.\\
\indent Let $a\equiv 2\mod 9$, $b_3\equiv -1\mod 3$, $v_3(a+7)\ge3$, $v_3(b+a+1)\ge 3$ and $2v_3(a+7)>v_3(a+b+1)+1$. Then $\phi_{1}$ provides one prime ideal of residual  degree $1$ and $\phi_{-1}$-Newton polygon of $f(x)$ has two edges of positive slope. The first edge is the line segment joinig the points $(4,~0)$ and $(5,~1)$ and the second edge join $(5,~1)$ with $(7,~v_3(b-a-1))$. The residual polynomial attached to the first edge is linear and to second edge is $(Y-\bar{1})(Y+\bar{1})$ or $Y^2+\bar{1}$ (according as $(b+a+1)_3\equiv -1\mod 3$ or $(b+a+1)_3\equiv 1\mod 3$). Thus either $\mathfrak{ J}=\mathfrak{ p}_1^3\mathfrak{ p}_2^2\mathfrak{p_3}$, where  the residual degree of $\mathfrak{p_i}=1$, for each $i=1,2,3$ or  $\mathfrak{ J}=\mathfrak{ p}_1^3\mathfrak{ p}_2^2\mathfrak{p_3}$, where  the residual degree of $\mathfrak{p_1}$, $\mathfrak{p_2}$   is $1$ and of $\mathfrak{p_3}$ is $2$.\\
\indent Let $a\equiv 2\mod 9$, $b_3\equiv -1\mod 3$, $v_3(a+7)\ge3$, $v_3(b+a+1)\ge 3$ and $2v_3(a+7)=v_3(a+b+1)+1$. Then $\phi_{1}$ provides one prime ideal of residual  degree $1$ and $\phi_{-1}$ has two edges of positive slope. The residual polynomial attached to the first edge is linear and to the second edge is $Y^2-Y-1$ or $Y^2+Y-1$. Thus  $\mathfrak{ J}=\mathfrak{ p}_1^3\mathfrak{ p}_2^2\mathfrak{p_3}$, where  the residual degree of $\mathfrak{p_1}$, $\mathfrak{p_2}$   is $1$ and of $\mathfrak{p_3}$ is $2$. So, $3\nmid i(K).$\\
\indent Let $a\equiv 2\mod 9$, $b_3\equiv -1\mod 3$, $v_3(a+7)\ge3$, $v_3(b+a+1)\ge 3$ and $2v_3(a+7)<v_3(a+b+1)+1$. Then $\phi_{-1}$ provides one prime ideal of residual  degree $1$ and $\phi_{1}$-Newton polygon of $f(x)$ has three edges of positive slope. The first edge is the line segment joining the points $(4,~0)$ and $(5,~1)$, the second edge join $(5,~1)$ with $(6,~v_3(a+7))$ and the third edge joins $(6,~v_3(a+7))$ and $(7,~v_3(b+a+1))$. The residual polynomial attached to the each edge is linear. Hence $\mathfrak{ J}=\mathfrak{ p}_1^3\mathfrak{ p}_2\mathfrak{p_3}\mathfrak{p_4}$, where  the residual degree of $\mathfrak{p_i}=1$, for each $i=1,2,3,4$  is $1$. Thus $3\mid i(K)$.\\
\indent Let $a\equiv 5\mod 9$ and $v_3(b)=1$, then $v_3(a+7)=1$. In this, either $\phi_{1}$ provides one prime ideal of residual degree $1$ and $\phi_{-1}$ provides two prime ideal of residual degree $1$ each or $\phi_{1}$ provides two prime ideal of residual degree $1$ each and $\phi_{-1}$ provides one prime ideal of residual degree $1$ (according as $b_3\equiv -1\mod 3$ or $b_3\equiv 1\mod 3$). Thus $\mathfrak{ J}=\mathfrak{ p}_1^3\mathfrak{ p}^2_2\mathfrak{p_3}$, where  the residual degree of $\mathfrak{p_i}$, for each $i=1,2,3$  is $1$. Thus $3\mid i(K)$.\\
\indent Let $v_3(a+1)>1$ and $v_3(b)=1$, then $v_3(a+7)=1$ and $v_3(b-\delta(a+1))=1$. Thus for  each $\delta$, $\phi_\delta$ provides one prime ideal of residual degree $1$. So $\mathfrak{J}=\mathfrak{ p}_1^3\mathfrak{ p}_2^3$.\\
 \indent Let $v_3(a+1)>1$, $v_3(b)>1$ and  $((a+1)_3,b_3)\in \{(-1,1), (1,-1)\}\mod3$, then $v_3(b-a-1)>1$ and $v_3(b+a+1)=1$. So the $\phi_{-1}$-Newton polygon of $f(x)$ has two edges of positive slope and residual polynomial attached to each edge is linear. And $\phi_1$ provides one prime ideal of degree $1$. Therefore $\mathfrak{ J}=\mathfrak{ p}_1^2\mathfrak{ p}_2\mathfrak{ p}_3^3$. Thus $3\mid i(K)$.\\
 \indent Let $v_3(a+1)>1$, $v_3(b)>1$  and $((a+1)_3,b_3)\in \{(-1,-1), (1,1)\}\mod3$, then $v_3(b+a+1)>1$ and $v_3(b-a-1)=1$. So the $\phi_{1}$-Newton polygon of $f(x)$ has two edges of positive slope and residual polynomial attached to each edge is linear. The $\phi_{-1}$ provides one prime ideal of degree $1$. Therefore $\mathfrak{ J}=\mathfrak{ p}_1^2\mathfrak{ p}_2\mathfrak{ p}_3^3$. Thus $3\mid i(K).$ \\\\
\noindent\textbf{Case B3:} $a\equiv 1\mod 3$ and $b\equiv 0\mod 3$. In this case $f(x)\equiv x(x^2+1)^3\mod 3$. In this case $x$-provides one prime ideal of residual degree $1$. If we denote $R(\mathfrak{(p)}$ the residual degree of the prime ideal $\mathfrak{p}$, then for $3\mathbb{O}_K$ we have the following possibilities.
\begin{center}
	\begin{longtable}[h!]{ |m{1cm}|m{4cm}|m{7cm}|m{7cm}|} 
		\hline
		Case &{ Factorization  of $3\mathbb{O}_K$ }& Residual degree\\
		\hline
		(1)& $3\mathbb{O}_K=\mathfrak{ p}_1\mathfrak{ p}_2^3$& $R(\mathfrak{p}_1)=1$ and $R(\mathfrak{p}_1)=2$\\
		\hline
		(2)& $3\mathbb{O}_K=\mathfrak{ p}_1\mathfrak{ p}_2\mathfrak{ p}_3$& $R(\mathfrak{p}_1)=1$, $R(\mathfrak{p}_1)=2$ and $R(\mathfrak{p}_1)=4$\\
		\hline
		(3)& $3\mathbb{O}_K=\mathfrak{ p}_1\mathfrak{ p}_2$& $R(\mathfrak{p}_1)=1$ and  $R(\mathfrak{p}_2)=6$ \\
		\hline
		(4)& $3\mathbb{O}_K=\mathfrak{ p}_1\mathfrak{ p}_2^2\mathfrak{ p}_3$& $R(\mathfrak{p}_1)=1$, $R(\mathfrak{p}_1)=2$ and $R(\mathfrak{p}_1)=2$\\
		\hline
		(5)& $3\mathbb{O}_K=\mathfrak{ p}_1\mathfrak{ p}_2\mathfrak{ p}_3\mathfrak{ p}_4$& $R(\mathfrak{p}_1)=1$ and  $R(\mathfrak{p}_i)=2, i=2,3,4$\\
		\hline
		\end{longtable}
	Table
\end{center}
Clearly in each case, there can be atmost $3$ prime ideals of residual degree $2$ each. But there are $5$ monic irreducible polynomial of degree $2$ over $\F_3.$ Thus in view of Lemma  \ref{Th 1.12} $3\nmid i(K)$.	This completes the proof.
\end{proof}
\begin{proof}[Proof of Theorem \ref{Th1.3}] 
	In this case $5\mid i(K)$ if and only if $a^7\equiv 2b^6\mod 5$.\\ i.e. $(a,~b)\in \{(0,~0),(2,~2),(2,~-2),(3,~1),(3,~-1)\}$.\\
	Let $(a,~b)=(0,~0)$, then $f(x)\equiv x^7\mod 7$. If $7v_5(a)>6v_5(b)$, then $x$-Newton polygon of $f(x)$ has a single edge of positive slope. The residual polynomial attached to this edge is linear. Therefore $5\mathbb{O}_K=\mathfrak{ p}^7$.  If $7v_5(a)<6v_5(b)$ and $v_5(a)\in\{1,5\}$, then $x$-Newton polygon of $f(x)$ has two  edges of positive slope. The residual polynomial attached to each edge is linear. Therefore $5\mathbb{O}_K=\mathfrak{ p}_1^6\mathfrak{ p}_2$. If $7v_5(a)<6v_5(b)$ and $v_5(a)\in\{2,4\}$, then $x$-Newton polygon of $f(x)$ has two  edges of positive slope. The residual polynomial attached to first edge is $(Y+\bar{1})(Y-\bar{1})$ or $ (Y+\bar{3})(Y-\bar{2})$ and to the second edge is linear. Therefore   $\mathfrak{ p}_1^3\mathfrak{ p}_2^3\mathfrak{ p}_3$. If $7v_5(a)<6v_5(b)$ and $v_5(a)\in\{2,4\}$, then $x$-Newton polygon of $f(x)$ has two  edges of positive slope. The residual polynomial attached to first edge is $(Y+\bar{1})(Y^2-Y+\bar{1})$ or $ (Y-\bar{1})(Y^2+Y+\bar{1})$ and to the second edge is linear. Therefore $5\mathbb{O}_K=\mathfrak{ p}_1^3\mathfrak{ p}_2^3\mathfrak{ p}_3$.  Thus $5\nmid i(K)$.\\
	Let $(a,~b)=(2,~2)$, then $f(x)\equiv (x+2)^2(x-1)(x^4+2x^3+4x^2+2x+2)\mod5$, so $5\mathbb{O}_K=\mathfrak{ p}_1\mathfrak{ p}_2\mathfrak{ p}_3$  or  $5\mathbb{O}_K=\mathfrak{ p}_1\mathfrak{ p}_2\mathfrak{ p}_3$. Thus $5\nmid i(K)$.\\
	Let $(a,~b)=(2,~-2)$, then $f(x)\equiv (x+1)(x+3)^2(x^4+3x^2+4x^2+3x+2)\mod5$, so $5\mathbb{O}_K=\mathfrak{ p}_1\mathfrak{ p}_2\mathfrak{ p}_3$  or $5\mathbb{O}_K=\mathfrak{ p}_1^2\mathfrak{ p}_2\mathfrak{ p}_3$ or   $5\mathbb{O}_K=\mathfrak{ p}_1\mathfrak{ p}_2\mathfrak{ p}_3\mathfrak{ p}_4$. Thus $5\nmid i(K)$.\\
	Let $(a,~b)=(2,~-2)$, then $f(x)\equiv (x+1)(x+3)^2(x^4+3x^3+4x^2+3x+2)\mod5$, so $5\mathbb{O}_K=\mathfrak{ p}_1\mathfrak{ p}_2\mathfrak{ p}_3$  or $\mathfrak{ p}_1\mathfrak{ p}_2^2\mathfrak{ p}_3$ $5\mathbb{O}_K=\mathfrak{ p}_1\mathfrak{ p}_2\mathfrak{ p}_3\mathfrak{ p}_4$. Thus $5\nmid i(K)$.\\
	Let $(a,~b)=(3,~1)$, then $f(x)\equiv (x+4)(x+3)^2(x^4+4x^3+x^2+x+2)\mod5$, so $5\mathbb{O}_K=\mathfrak{ p}_1\mathfrak{ p}_2\mathfrak{ p}_3$  or $\mathfrak{ p}_1\mathfrak{ p}_2^2\mathfrak{ p}_3$ $5\mathbb{O}_K=\mathfrak{ p}_1\mathfrak{ p}_2\mathfrak{ p}_3\mathfrak{ p}_4$. Thus $5\nmid i(K)$.\\
	Let $(a,~b)=(3,~-1)$, then $f(x)\equiv (x+1)(x+2)^2(x^4+x^2+x^2+4x+2)\mod5$, so $5\mathbb{O}_K=\mathfrak{ p}_1\mathfrak{ p}_2\mathfrak{ p}_3$  or $\mathfrak{ p}_1\mathfrak{ p}_2^2\mathfrak{ p}_3$ $5\mathbb{O}_K=\mathfrak{ p}_1\mathfrak{ p}_2\mathfrak{ p}_3\mathfrak{ p}_4$. Thus $5\nmid i(K)$.\\
	Next we show that $7\nmid i(K)$. Suppose there exists distinct non-zero prime ideals $\mathfrak{p}_1$,  $\cdots$, $\mathfrak{p}_r$ of $\mathbb{O}_K$ such that $7\mathbb{O}_K=\mathfrak{p}_1^{e_1}\cdots\mathfrak{p}_r^{e_r}$, where $e_i \geq 1$, then by the Fundamental Equality, $e_1f_1+\cdots +e_rf_r=7 $ where $f_i$ is the residual degree of $\mathfrak p_i$ for $i=1,2,\cdots,r$. Since $e_i\ge1$ for all $i=1,2,\cdots,r$, therefore there can be at most $7$ prime ideals lying above $7$, but for every positive integer $h$ the number of monic  irreducible polynomials of degree $h$ in $\F_7[x]$ is greater than or equal to $7$. So by Lemma \ref{Th 1.12}, $7\nmid i(K)$. This completes the proof. \end{proof}

\vspace{0.2 cm}


\end{document}